\documentclass[leqno]{amsart}
\usepackage{amsmath}
\usepackage{amssymb}
\usepackage{amsthm}
\usepackage{enumerate}
\usepackage[mathscr]{eucal}
\usepackage{graphicx}
\usepackage{color}
\theoremstyle{plain}
\newtheorem{theorem}{Theorem}[section]

\newtheorem{lemma}{Lemma}[section]

\theoremstyle{definition}

\usepackage[pagewise]{lineno}

\begin{document}
	
	\title[A study of symmetric points in Banach spaces]{A study of symmetric points in Banach spaces}
	\author[Sain, Roy, Bagchi and Balestro]{Debmalya Sain, Saikat Roy, Satya Bagchi, Vitor Balestro}
                \newcommand{\acr}{\newline\indent}

\address[Sain]{Department of Mathematics\\ Indian Institute of Science\\ Bengaluru 560012\\ Karnataka \\INDIA\\ }
\email{saindebmalya@gmail.com}

\address[Roy]{Department of Mathematics\\ National Institute of Technology Durgapur\\ Durgapur 713209\\ West Bengal\\ INDIA}
\email{saikatroy.cu@gmail.com}

\address[Bagchi]{Department of Mathematics\\ National Institute of Technology Durgapur\\ Durgapur 713209\\ West Bengal\\ INDIA}
\email{satya.bagchi@maths.nitdgp.ac.in}

\address[Balestro]{Instituto de Mathem$\acute{A}$tica e Estat$\acute{I}$stica\\Universidade Federal Fluminense\\ 24210201 Niter$\acute{O}$i\\BRAZIL}
\email{vitorbalestro@id.uff.br}

\thanks{The research of Dr. Debmalya Sain is sponsored by Dr. D. S. Kothari Postdoctoral Fellowship under the
mentorship of Professor Gadadhar Misra. The research of Mr. Saikat Roy is supported by CSIR MHRD in terms of Junior research Fellowship.} 


\begin{abstract}
We completely characterize the left-symmetric points, the right-symmetric points, and, the symmetric points in the sense of Birkhoff-James, in a Banach space. We obtain a complete characterization of the left-symmetric (right-symmetric) points in the infinity sum of two Banach spaces, in terms of the left-symmetric (right-symmetric) points of the constituent spaces. As an application of this characterization, we explicitly identify the left-symmetric (right-symmetric) points of some well-known three-dimensional polyhedral Banach spaces. 
\end{abstract}

\maketitle
\section{Introduction.} 

The purpose of this short note is to study the left-symmetric points and the right-symmetric points of a Banach space. We also explore the left-symmetric points and right-symmetric points of the infinity sum of two Banach spaces and apply the obtained results to certain polyhedral Banach spaces. Let us first establish the relevant notations and the terminologies to be used throughout the article. \\

Throughout the text, we use the symbols $ \mathbb{X}, \mathbb{Y}, \mathbb{Z} $ to denote Banach spaces. In this article we work with only \textit{real} Banach spaces. Let $ \mathbb{X} \bigoplus_{\infty} \mathbb{Y} $ denote the infinity sum of $ \mathbb{X} $ and $ \mathbb{Y}, $ endowed with the norm $ \| (x,y) \| = \max\{\|x\|,\|y\|\} $ for any $ (x,y) \in \mathbb{X} \bigoplus_{\infty} \mathbb{Y}. $ Since there is no chance of confusion, we use the same symbol $ \|~\| $ to denote the norm function in $ \mathbb{X},\mathbb{Y} $ and $ \mathbb{X} \bigoplus_{\infty} \mathbb{Y}. $ Similarly, we use the same symbol $ 0 $ to denote the zero vector of any Banach space. Given $ x \in \mathbb{X} $ and $ r>0, $ let $ B(x,r) $ denote the open ball with center at $ x $ and radius $ r. $ Let $ B_{\mathbb{X}} = \{x \in \mathbb{X} \colon \|x\| \leq 1\}$  and $ S_{\mathbb{X}} = \{x \in \mathbb{X} \colon \|x\|=1\} $ be the unit ball and the unit sphere of $\mathbb{X}$ respectively. We say that $ \mathbb{X} $ is polyhedral if $ B_{\mathbb{X}} $ have only finitely many extreme points. Let $ E_{\mathbb{X}}$ denote the set of all extreme points of the unit ball $ B_{\mathbb{X}} $. \\

A vector  $x \in {\mathbb{X}}$ is said to be \textit{Birkhoff-James orthogonal} to a vector $y \in \mathbb{X}$, written as $x \perp_B y$, if $ \|x+\lambda y\|\geq\|x\|$ for all $ \lambda \in \mathbb{R}.$ (see \cite{AMW} and \cite{B}) Clearly, this definition is applicable in any Banach space, in particular, in the infinity sum of two given Banach spaces. It is easy to see that Birkhoff-James orthogonality is homogeneous, i.e., $ x \perp_B y $ implies that $ \alpha x \perp_B \beta y, $ for any two scalars $ \alpha,\beta \in \mathbb{R}. $ However, since Birkhoff-James orthogonality is not symmetric in general, i.e., $ x \perp_B y $ does not necessarily imply that $ y \perp_B x, $ the notions of left-symmetric points and the right-symmetric points were introduced in \cite{Sa}. Let us recall the relevant definitions here. An element $ x \in \mathbb{X} $ is said to be a left-symmetric point in $ \mathbb{X} $ if for any $ y \in \mathbb{X}, $ $ x \perp_B y $ implies that $ y \perp_B x. $ On the other hand, $ x \in \mathbb{X}  $ is said to be a right-symmetric point in $ \mathbb{X} $ if for any $ y \in \mathbb{X}, $ $ y \perp_B x $ implies that $ x \perp_B y.$ $ x \in \mathbb{X}  $ is said to be a symmetric point in $ \mathbb{X} $ if $ x $ is both left-symmetric and right-symmetric.\\

In this article, we characterize the left-symmetric points and the right-symmetric points in a given Banach space. In order to obtain the required characterization, we require the following two notions introduced in \cite{Sa}. Given any vector $ x \in \mathbb{X}, $ we define the sets $x^+$ and $x^-$ by the following: we say that $ y \in x^{+} $ if $ \| x+\lambda y \| \geq \| x \| $ for all $ \lambda \geq 0 $, and we have that $ y \in x^{-} $ if $ \| x+\lambda y \| \geq \| x \| $ for all $ \lambda \leq 0. $ Basic properties of these two notions, their connection with Birkhoff-James orthogonality and some applications of it have been explored  in \cite{S}. In particular, it is easy to see that the set of Birkhoff-James orthogonal vectors of $x$ is given in terms of $x^+$ and $x^-$ by $ x^{\perp} := \{ y \in \mathbb{X}:~x \perp_B y \} = x^{+} \cap x^{-}. $ As an application of the results obtained by us, we explicitly identify the left-symmetric points and the right-symmetric points of a three-dimensional polyhedral Banach space whose unit ball is a right prism whose base is a regular $ 2n $-gon. In this direction, the first step is to characterize the left-symmetric points and the right-symmetric points of the analogous two-dimensional space, whose unit ball is a regular $ 2n $-gon. We would like to remark that in the two-dimensional case, the notions of left-symmetric points and right-symmetric points may be viewed as a local generalization of Radon planes \cite{R}, i.e., two-dimensional normed planes where Birkhoff-James orthogonality is symmetric. We refer the readers to \cite{BMT,H,MS} for some of the extensive studies on Radon planes. It has been proved in \cite{H} that if $ \mathbb{X} $ is a two-dimensional polyhedral Banach space whose unit ball is a regular $ 2n $-gon then $ \mathbb{X} $ is a Radon plane if and only if $ n $ is odd. Therefore, it is perhaps not surprising that in the three-dimensional case of the unit ball being a right prism with regular $ 2n $-gon as its base, the description of the left-symmetric points and the right-symmetric points depends on whether $ n $ is even or odd.

\section{Main Results}
We begin with a complete characterization of the left-symmetric points of a Banach space. 

\begin{theorem}\label{left-symmetric characterization}
	Let $ \mathbb{X} $ be a real Banach space. Then $ x \in \mathbb{X} $ is a left-symmetric point in $ \mathbb{X} $ if and only if given any $ u \in \mathbb{X}, $ the following two conditions hold true:\\ 
	(i) $ u \in x^{-} $ implies that $ x \in u^{-}, $\\
	(ii) $ u \in x^{+} $ implies that $ x \in u^{+}. $
	
\end{theorem}

\begin{proof}

Let us first prove the sufficient part of the theorem. Let $ u \in \mathbb{X} $ be such that $ x \perp_B u. $ Equivalently, $ u \in x^{+} \cap x^{-}. $ Suppose by contradiction that $ u \not\perp_B x. $ Then either $ x \notin u^{+} $ or $ x \notin u^{-}. $ In either case, we arrive at a contradiction to our initial assumption. This completes the proof of the sufficient part of the theorem. We now prove the necessary part of the theorem. We prove $ (i) $  and note that the proof of $ (ii) $ can be completed similarly. Let us assume that $ u \in x^{-} $. Without loss of generality, let $ \| u \| = 1 $. Now, there are two possible cases:\\

Case I: Let $ u\in x^{+} $. In that case, we have that $ x \perp_B u. $ Since $ x $ is a left-symmetric point, we have $ u\perp_B x $. In particular, $ x\in u^{-} $.\\

Case II: Let $ u\notin x^{+} $. Then it follows from Proposition $ 2.1 $ of \cite{S} that $ u\in x^{-}\setminus x^{\perp} $. Let $ \mathcal{A}=\{\alpha x+\beta u : \|\alpha x+\beta u\|=1, \alpha\geq 0,\beta\geq 0\} $. In other words, $ \mathcal{A} $ is the closed shorter arc on $ S_\mathbb{X}\cap span\{ u, x\} $ between $ u $ and $ x $. Let $\mathcal{B}=\{\beta : \exists \ \alpha\in\mathbb{R} \ \mathrm{such \ that} \ \alpha x+\beta u\in \mathcal{A} \ \mathrm{and} \ \ (\alpha x+\beta u)\in x^{-}\setminus x^{\perp} \} $. Clearly, $ 1\in \mathcal{B} $ and $ \mathcal{B} $ is a bounded subset of  $ \mathbb{R} $. Therefore,  $\inf\mathcal{B} $ exists. Let $ \inf\mathcal{B}= \beta_0 $. Then from definition of infimum, there exists a sequence $ \{ \alpha_n x+\beta_n u\}\in \mathcal{A} $ such that $ \{ \alpha_n x+\beta_n u\}\in x^{-}~\setminus~ x^{\perp} $ and $ \beta_n \longrightarrow \beta_0 $. Since $ \{\alpha_n\} $ is clearly bounded, we get that it has a convergent subsequence. Hence, passing to such a subsequence if necessary we may assume that $ \alpha_n  \longrightarrow \alpha_0 \geq 0 $. Then $ \{ \alpha_n x + \beta_n u \} \longrightarrow \alpha_0 x + \beta_0 u. $ Moreover, since $ \mathcal{A} $ is closed, $ \alpha_0 x + \beta_0 u \in \mathcal{A}. $  Now, it follows from the continuity of the norm function that if $ z \in S_{\mathbb{X}} $ is such that $ y \in z^{-}~\setminus~z^{\perp} $ then there exists $ \delta > 0 $ so that for every $ w \in B(z,\delta), $ we have that $ y \in w^{-}~\setminus~w^{\perp}. $ Therefore, we obtain that $ \alpha_0 x + \beta_0 u \notin x^{-}~\setminus~x^{\perp} $, as otherwise the definition of infimum will be violated from the previous observation. On the other hand, since $ (\alpha_n x + \beta_n u)\in x^{-}~\setminus~x^{\perp}, $ once again it follows from the continuity of the norm function that $ (\alpha_0 x + \beta_0 u)\in x^{-}. $ Combining these two observations, it is now easy to deduce that $ (\alpha_0 x + \beta_0 u)\in x^{\perp}. $ We note that $ 0<\beta_0 \leq 1$. Suppose by contradiction that $ x\notin u^{-} $. We claim that it follows from $ x\notin u^{-} $ that $ (\alpha_0 x + \beta_0 u)\not\perp_B x $. In order to prove our claim, we note that $ (\alpha_0 x + \beta_0 u)\perp_B x $ implies that $ \|(\alpha_0 x + \beta_0 u)+\lambda x\| \geq \|(\alpha_0 x + \beta_0 u)\|, $ \, for all $ \lambda\in \mathbb{R} $. Therefore, $  \|\frac{(\alpha_0 + \lambda)}{\beta_0}x + u \|  \geq \frac{1}{\beta_0}\|(\alpha_0 x + \beta_0 u)\| = \frac{1}{\beta_0} \geq \| u\| $ for all $ \lambda \in \mathbb{R} $ (here, we recall that $\alpha_0x + \beta_0u \in \mathcal{A}$, meaning it is a unit vector, and we also remember that we are assuming that $\|u\| = 1$). Taking $ \frac{(\alpha_0 + \lambda)}{\beta_0} = \mu $, we have that  $ \|\mu x + u \|  \geq \|u\| $ for all $ \mu\in \mathbb{R} $, i.e.,  $ u\perp_B x $. In other words, $ x\in u^{+}\cap u^{-} $, which is in contradiction to $ x\notin u^{-} $. It follows that we have indeed that $(\alpha_0x+\beta_0u)\not\perp_B x$. However, since $ x\perp_B (\alpha_0 x+ \beta_0 u), $ this contradicts the fact that $ x $ is a left-symmetric point in $ \mathbb{X} $. Consequently, we must have that $ x\in u^{-} $. This completes the proof of the necessary part of the theorem and thereby establishes it completely.
\end{proof}
  
Our next aim is to completely characterize the right-symmetric points of a Banach space. 

\begin{theorem}\label{right-symmetric characterization}
Let $ \mathbb{X} $ be a real Banach space. Then $ x \in \mathbb{X} $ is a right-symmetric point in $ \mathbb{X} $ if and only if given any $ u \in \mathbb{X}, $ the following two conditions hold true:\\ 
	(i) $ x \in u^{-} $ implies that $ u \in x^{-}, $\\
	(ii) $ x \in u^{+} $ implies that $ u \in x^{+}. $
\end{theorem}

\begin{proof}

Let us first prove the sufficient part of the theorem. Let $ u \in \mathbb{X} $ be such that $ u \perp_B x. $ Equivalently, $ x \in u^{+} \cap u^{-}. $ Suppose by contradiction that $ x \not\perp_B u. $ Then either $ u \notin x^{+} $ or $u \notin x^{-}. $ In either case, we arrive at a contradiction to our initial assumption. This completes the proof of the sufficient part of the theorem. We now prove the necessary part of the theorem. We  prove $ (i) $ and note that the proof of $ (ii) $ can be completed similarly. Let us assume that $x \in u^-$. Now, there are two possible cases:\\

Case I: Let $ x \in u^{+}. $ In that case, we have that $ u \perp_B x. $ Since $ x $ is a right-symmetric point in $ \mathbb{X}, $ it follows that $ x \perp_B u. $ In particular, $ u \in x^{-}. $ \\

Case II: Let $ x \notin u^{+}. $ Then it follows from Proposition $ 2.1 $ of \cite{S} that $ x \in u^{-}~\setminus~u^{\perp}. $ Suppose by contradiction that $ u \notin x^{-}. $ As in the proof of the previous theorem, let $ \mathcal{A} = \{ \alpha u + \beta x:~\| \alpha u + \beta x \|=1,~\alpha \geq 0, \beta \geq 0 \} $, and recall also that $ \mathcal{A} $ is the closed shorter arc on $ S_\mathbb{X}\cap span\{ u, x\} $ between $ u $ and $ x $. Let $ \mathcal{B}=\{ \beta:\exists \ \alpha \in \mathbb{R} \ \mathrm{such \ that} \ \alpha u + \beta x \in \mathcal{A} \ \mathrm{and} \ x \in (\alpha u + \beta x)^{-} ~\setminus~ (\alpha u + \beta x)^{\perp} \}. $ It is easy to check that $ \mathcal{B} $ is a bounded subset of $ \mathbb{R} $ and $ 0 \in \mathcal{B}. $ Let $ \beta_0 = \sup \mathcal{B}. $ It follows from the definition of supremum that there exists a sequence $ \{ \alpha_n u + \beta_n x \} $ in $ \mathcal{A} $ such that $ x \in (\alpha_n u + \beta_n x)^{-}~\setminus ~ (\alpha_n u + \beta_n x)^{\perp} $ and $ \beta_n \longrightarrow \beta_0. $ Since $ \{\alpha_n\} $ is clearly bounded, we get that it has a convergent subsequence. Hence, passing to such a subsequence if necessary we may assume that $ \alpha_n  \longrightarrow \alpha_0 \geq 0 $. Then $ \{ \alpha_n u + \beta_n x \} \longrightarrow \alpha_0 u + \beta_0 x. $ Moreover, it is easy to see that $ \alpha_0 u + \beta_0 x \in \mathcal{A}. $ Now, it follows from the continuity of the norm function that if $ z \in S_{\mathbb{X}} $ is such that $ x \in z^{-}~\setminus~z^{\perp} $ then there exists $ \delta > 0 $ so that for every $ w \in B(z,\delta), $ we have that $ x \in w^{-}~\setminus~w^{\perp}. $ Therefore, by using the definition of supremum, we obtain that $ x \notin (\alpha_0 u + \beta_0 x)^{-}~\setminus~(\alpha_0 u + \beta_0 x)^{\perp}. $ On the other hand, since $ x \in (\alpha_n u + \beta_n x)^{-}, $ once again it follows from the continuity of the norm function that $ x \in (\alpha_0 u + \beta_0 x)^{-}. $ Combining these two observations, it is now easy to deduce that $ x \in (\alpha_0 u + \beta_0 x)^{\perp}. $ We note that $ \alpha_0, \beta_0 > 0. $ In Theorem $ 2.1 $ of \cite{SPM}, taking $ \epsilon = 0 $ it is easy to see that $ x^{\perp} = x^{+} \cap x^{-} $  is a \emph{normal cone} in $ \mathbb{X}. $ We recall that a subset $ K $ of $ \mathbb{X} $ is said to be a \emph{normal cone} in $ \mathbb{X} $ if \\

\noindent$ (1)~ K + K \subset K,$

\noindent $(2)~ \alpha K \subset K $ for all $ \alpha \geq 0 $ and

\noindent $(3)~ K \cap (-K) = \{0\}. $\\

 Since $ u \in x^{+}~\setminus~x^{-}, $ it follows that if $ v \in S_{\mathbb{X}} $ is such that $ x \perp_B v, $ then $ v $ must be of the form $ v = \gamma_0 u + \delta_0 x, $ where either of the following is true:\\

\noindent (a) $ \gamma_0 > 0 $ and $ \delta_0 < 0, $

\noindent (b)  $ \gamma_0 < 0 $ and $ \delta_0 > 0. $\\

In particular, it follows that $ x \not\perp_B (\alpha_0 u + \beta_0 x), $ since $ \alpha_0, \beta_0 > 0. $ However, since $ (\alpha_0 u + \beta_0 x) \perp_B x, $ this contradicts our assumption that $ x $ is a  right-symmetric point in $ \mathbb{X}. $ Therefore, it must be true that $ u \in x^{-}. $ This completes the proof of the necessary part of the theorem and thereby establishes it completely.
\end{proof}

We note that as a consequence of Theorem \ref{left-symmetric characterization} and Theorem \ref{right-symmetric characterization}, it is possible to completely characterize symmetric points of a Banach space. We record this useful observation in the next theorem. The proof of the theorem is omitted as it is now obvious.

\begin{theorem}\label{symmetric characterization}
Let $ \mathbb{X} $ be a real Banach space. Then $ x \in \mathbb{X} $ is a symmetric point in $ \mathbb{X} $ if and only if given any $ u \in \mathbb{X}, $ the following four conditions hold true:\\
    (i) $ u \in x^{-} $ implies that $ x \in u^{-}, $\\
	(ii) $ u \in x^{+} $ implies that $ x \in u^{+}, $\\	 
	(iii) $ x \in u^{-} $ implies that $ u \in x^{-}, $\\
	(iv) $ x \in u^{+} $ implies that $ u \in x^{+}. $
	
\end{theorem}

We now turn our attention towards obtaining an explicit characterization of the left-symmetric points and the right-symmetric points of the infinity sum of two given Banach spaces. Since Birkhoff-James orthogonality is homogeneous, it is clearly sufficient to describe the left-symmetric points and the right-symmetric points on the unit sphere of the concerned space. First we characterize the left-symmetric points of the infinity sum of two Banach spaces.

\begin{theorem}\label{left-symmetric}
	Let $\mathbb{X}, \mathbb{Y}$ be real Banach spaces and let $ \mathbb{Z}=\mathbb{X} \bigoplus_{\infty} \mathbb{Y}. $ Then $ (x,y) \in S_{\mathbb{Z}} $ is a left-symmetric point in $ \mathbb{Z} $ if and only if either of the following is true:\\
	(i) $ x \in S_{\mathbb{X}} $ is a left-symmetric point in $ \mathbb{X} $ and $ y=0, $\\
	(ii) $ x=0 $ and $ y \in S_{\mathbb{Y}} $ is a left-symmetric point in $ \mathbb{Y}. $ 

\end{theorem}

\begin{proof}
Let us first prove the sufficient part of the theorem. We first assume that $ (i) $ holds true. Let $ (x,0) \perp_B (u,v) $ in $ \mathbb{X} \bigoplus_{\infty} \mathbb{Y}, $ i.e., $ \max\{ \| x+\lambda u \|, \| \lambda v \| \} \geq \| x \| = 1. $ This shows that for any $ \lambda $ with sufficiently small absolute value, $ \| x+\lambda u \| \geq \| x \|. $ Using the convexity of the norm function, it is now easy to verify that $ \| x+\lambda u \| \geq \| x \| $ for any $ \lambda, $ i.e., $ x \perp_B u. $ Since $ x $ is a left-symmetric point in $ \mathbb{X}, $ it follows that $ u \perp_B x. $ Now, for any $ \lambda \in \mathbb{R}, $ we have the following:
\[ \| (u,v)+\lambda (x,0) \| =\max\{ \| u+\lambda x \|, \| v \| \} \geq \max\{\|u\|,\|v\|\} = \| (u,v) \|. \]
In other words, $ (u,v) \perp_B (x,0)$ whenever $ (x,0) \perp_B (u,v). $ This proves that $ (x,0) $ is a left-symmetric point in $ \mathbb{Z}. $ Similarly, assuming that $ (ii) $ holds true, we can show that $ (0,y) $ is a left-symmetric point in $ \mathbb{Z} $. This completes the proof of the sufficient part of the theorem. Let us now prove the necessary part of the theorem. Let $ (x,y) \in S_{\mathbb{Z}} $ be a left-symmetric point in $ \mathbb{Z}. $ Clearly, either $ \| x \| = 1 $ or $ \| y \| =1. $ Without loss of generality, let us assume that $ \| x \| = 1. $ We note that $ \| y \| \leq 1. $ We claim that $ y=0. $ Suppose by contradiction that $ y \neq 0. $ Since $ \max\{ \| x \|, \| y+\lambda y \| \}\geq \| x \|=\| (x,y) \| =1 , $ it follows that $ (x,y) \perp_B (0,y). $ However, for any $ \lambda < 0 $ with sufficiently small absolute value, we have that $ \| (0,y)+\lambda (x,y) \| = \max\{ \| \lambda x \|, \| y+\lambda y \| \} < \| y \| = \| (0,y) \|. $ In other words, $ (0,y) \not\perp_B (x,y). $ Clearly, this contradicts our assumption that $ (x,y) $ is a left-symmetric point in $ \mathbb{Z}. $ So, we conclude that if $ (x,y) \in S_{\mathbb{Z}} $ is a left-symmetric point in $ \mathbb{Z} $ with $ \| x \| = 1 $ then $ y = 0 $. In that case, we also claim that $ x $ is a left-symmetric point in $ \mathbb{X}. $ Suppose by contradiction that there exists $ u \in S_{\mathbb{X}} $ such that $ x \perp_B u $ but $ u \not\perp_B x. $ It is easy to see that either $ x \notin u^{+} $ or $ x \notin u^{-}. $ Without loss of generality, let us assume that $ x \notin u^{+}. $ Then $ \| u+\lambda x \| < \| u \|=1, $ for some $ \lambda >0 $. Using this fact, it is easy to verify that $ (x,0) \perp_B (u,0) $ but $ (u,0) \not\perp_B (x,0). $ Once again, this contradicts our assumption that $ (x,0) $ is a left-symmetric point in $ \mathbb{Z}. $ This contradiction completes the proof of our claim that $ x $ is a left-symmetric point in $ \mathbb{X}. $ Thus we see that all the conditions stated in $ (i) $ are satisfied. Similarly, assuming that $ \| y \|=1 , $ we can show that all the conditions stated in $ (ii) $ hold true. This completes the proof of the necessary part of the theorem and establishes it completely. 	
\end{proof}

We next characterize the right-symmetric points of the infinity sum of two Banach spaces. We require the following lemma to serve our purpose.

\begin{lemma}\label{right-symmetric lemma}
Let $\mathbb{X}, \mathbb{Y}$ be  real Banach spaces and let $ \mathbb{Z} = \mathbb{X} \bigoplus_{\infty} \mathbb{Y}.  $ Let $ (u,v) \in S_{\mathbb{Z}} $ be such that $ \| u \| = \| v \| =1. $ If $ (x,y) \in \mathbb{Z} $ is such that $ (u,v) \perp_B (x,y) $ then either of the following is true:\\
	(i) $ x \in u^{+} $ and $ y \in v^{-}, $\\
	(ii) $ x \in u^{-} $ and $ y \in v^{+}. $ 
	
\end{lemma}

\begin{proof}
It follows from Proposition $ 2.1 $ of \cite{S} that either $ x \in u^{+} $ or $ x \in u^{-}. $ Without loss of generality, let us assume that $ x \in u^{+}. $ Let us consider the following two cases:\\

Case I: Let $ x \in u^{-}. $ By the same argument, either $ y \in v^{+} $ or $ y \in v^{-}. $ In each of these two sub-cases, it is easy to see that the statement of the lemma holds true.\\

Case II: Let $ x \notin u^{-}. $ Then $ \| u + \lambda x \| < \| u \| = 1, $ whenever $ \lambda < 0 $ is of sufficiently small absolute value. Since $ (u,v) \perp_B (x,y), $ we have that $ \| (u,v) + \lambda (x,y) \| = \max\{ \| u + \lambda x \|, \| v + \lambda y \| \} \geq \| (u,v) \| = 1 $ for any $ \lambda \in \mathbb{R}. $ Therefore, whenever $ \lambda < 0 $ is of sufficiently small absolute value, we must have that $ \| v + \lambda y \| \geq \| v \| = 1. $ Using the convexity of the norm function, it is now easy to deduce that $ \| v + \lambda y \| \geq \| v \| = 1 $ for all $ \lambda < 0. $ In other words, $ y \in v^{-}. $ This completes the proof of the lemma.
 \end{proof}

Now the promised characterization:

\begin{theorem}\label{right-symmetric}
Let $\mathbb{X}, \mathbb{Y}$ be real Banach spaces and let $ \mathbb{Z}=\mathbb{X} \bigoplus_{\infty} \mathbb{Y}. $ Then $ (x,y) \in S_{\mathbb{Z}} $ is a right-symmetric point in $ \mathbb{Z} $ if and only if $ \| x \| = \| y \| = 1 $ and $ x, y $ are right-symmetric points in $ \mathbb{X} $ and $ \mathbb{Y} $ respectively.
	
\end{theorem}

\begin{proof}
Let us first prove the sufficient part of the theorem. Let $ (u,v) \in \mathbb{Z} $ be such that $ (u,v) \perp_B (x,y). $ Let us consider the following two cases:\\

Case I: Let $ \| u \| \neq \| v \|. $ Without loss of generality, we assume that $ \| u \| > \| v \|. $ Since $ (u,v) \perp_B (x,y), $ we have that $ \max\{ \| u + \lambda x \|, \| v + \lambda y \| \} \geq \| (u,v) \| = \| u \| $ for all $ \lambda \in \mathbb{R}. $ Therefore, whenever $ \lambda \in \mathbb{R} $ is of sufficiently small absolute value, it is easy to see that $ \| u + \lambda x \| \geq \| u \|. $ As argued before, using the convexity of the norm function, it can be proved that $ \| u + \lambda x \| \geq \| u \| $ for all $ \lambda \in \mathbb{R}, $ i.e., $ u \perp_B x. $ Since $ x $ is a right-symmetric point in $ \mathbb{X}, $ we have that $ x \perp_B u. $ Thus we can deduce the following for any $ \lambda \in \mathbb{R}: $
\[ \| (x,y) + \lambda (u,v) \| = \max\{ \| x + \lambda u \|, \| y + \lambda v \| \} \geq \| x + \lambda u \| \geq \| x \| = \| (x,y) \| = 1.  \]
This proves that $ (x,y) \perp_B (u,v). $\\

Case II: Let $ \| u \| = \| v \|. $ If $ \| u \| = \| v \| = 0 $ then it is immediate that $ (x,y) \perp_B (0,0) = (u,v). $ Without loss of generality, let us assume that $ \| u \| = \| v \| = 1. $ Applying Lemma \ref{right-symmetric lemma}, we observe that either of the following must hold true:\\
$ (i) $ $ x \in u^{+}  $ and $ y \in v^{-}, $ $ (ii) $ $ x \in u^{-} $ and $ y \in v^{+}. $ Let us first assume that $ (i) $ holds true. Since $ x, y $ are right-symmetric points in $ \mathbb{X} $ and $ \mathbb{Y} $ respectively, applying Theorem \ref{right-symmetric characterization}, it follows that $ u \in x^{+} $ and $ v \in y^{-}. $ Now, for any $ \lambda \geq 0, $ the following holds true:
\[ \| (x,y) + \lambda (u,v) \| = \max\{ \| x + \lambda u \|, \| y + \lambda v \| \} \geq \| x + \lambda u \| \geq \| x \| = \| (x,y) \| = 1.  \]
On the other hand, for any $ \lambda \leq 0, $ the following holds true:
\[ \| (x,y) + \lambda (u,v) \| = \max\{ \| x + \lambda u \|, \| y + \lambda v \| \} \geq \| y + \lambda v \| \geq \| y \| = \| (x,y) \| = 1.  \]
This proves that $ (x,y) \perp_B (u,v). $ Similarly, assuming that $ (ii) $ holds true, we can deduce the same conclusion. Therefore, $ (x,y) \perp_B (u,v), $ whenever $ (u,v) \perp_B (x,y). $ In other words, $ (x,y) $ is a right-symmetric point in $ \mathbb{Z}. $ This completes the proof of the sufficient part of the theorem.\\

Let us now prove the necessary part of the theorem. Let $ (x,y) \in S_{\mathbb{Z}} $ be a right-symmetric point in $ \mathbb{Z}. $ Clearly, $ \| y \| \leq 1. $ We claim that $ \| y \| = 1. $ Suppose by contradiction that $ \| y \| < 1. $ We note that under this assumption, we have that $ \| x \| = 1. $ Choose $ w \in S_{\mathbb{Y}} $ such that $ y \in w^{+}. $ Now, for any $ \lambda \geq 0, $ we have that $ \| (-x,w) + \lambda (x,y) \| \geq \| w + \lambda y \| \geq \| w \| = 1 = \| (-x,w) \|. $ On the other hand, for any $ \lambda < 0, $ we have that $ \| (-x,w) + \lambda (x,y) \| \geq \| -x + \lambda x \| = 1 - \lambda > \| (-x,w) \|. $ This proves that $ (-x,w) \perp_B (x,y). $ However, when $ \lambda > 0 $ is of sufficiently small absolute value, we have that $ \| (x,y) + \lambda (-x,w) \| = \max\{ \| x - \lambda x  \|, \| y + \lambda w \| \} = \| (1 - \lambda) x \| = 1 - \lambda < 1 = \| (x,y) \|. $ In particular, this implies that $ (x,y) \not\perp_B (-x,w). $ This contradicts our assumption that $ (x,y) $ is a right-symmetric point in $ \mathbb{Z}. $ Therefore, it follows that $ \| y \| = 1. $ Similarly, we can show that $ \| x \| = 1. $ Our next claim is that $ x $ is a right-symmetric point in $ \mathbb{X}. $ Suppose by contradiction that $ x $ is not a right-symmetric point in $ \mathbb{X}. $ Then there exists $ u \in S_{\mathbb{X}} $ such that $ u \perp_B x $ but $ x \not\perp_B u. $ It follows from Proposition $ 2.1 $ of \cite{S} that either $ u \in x^{+}~\setminus~x^{-} $ or $ u \in x^{-}~\setminus~x^{+}. $ Without loss of generality, let us assume that $ u \in x^{+}~\setminus~x^{-}. $ Choose $ v \in S_{\mathbb{Y}} $ such that $ v \in y^{+}~\setminus~y^{-}. $ Then there exists $ \lambda_0 < 0 $ such that $ \| x + \lambda_0 u \| < \| x \| = 1 $ and $ \| y + \lambda_0 v \| < \| y \| = 1. $ This shows that $ (x,y) \not\perp_B (u,v). $ On the other hand, since $ u \perp_B x, $ it is easy to check that $ (u,v) \perp_B (x,y). $ Clearly, this contradicts our assumption that $ (x,y) $ is a right-symmetric point in $ \mathbb{Z}. $ Therefore, $ x $ must be a right-symmetric point in $ \mathbb{X}. $ Similarly, we can show that $ y $ is a right-symmetric point in $ \mathbb{Y}. $ This completes the proof of the necessary part of the theorem and establishes it completely. 
\end{proof}

We would like to apply Theorem \ref{left-symmetric} and Theorem \ref{right-symmetric} to identify the left-symmetric points and the right-symmetric points of some three-dimensional polyhedral Banach spaces. However, first we need to consider the analogous two-dimensional polyhedral Banach spaces. It was observed by Heil \cite{H} that if $ \mathbb{X} $ is a two-dimensional real polyhedral Banach space such that $ S_{\mathbb{X}} $ is a regular polygon with $ 2n $ sides, then Birkhoff-James orthogonality is symmetric in $ \mathbb{X} $ if and only if $ n $ is odd. In that case, every element of $ \mathbb{X} $ is a left-symmetric point as well as a right-symmetric point in $ \mathbb{X}. $  In the next two theorems, we characterize the left-symmetric points and the right-symmetric points of a two-dimensional real polyhedral Banach space $ \mathbb{X}, $ when $ S_{\mathbb{X}} $ is a regular polygon with $ 2n $ sides, where $ n $ is even. In this context, let us mention the following notations which are relevant to our next two theorems. Let $ \mathcal{N} =\{ 0, 1, \dots , 4n-1 \} $. We equip $ \mathcal{N} $ with a binary operation $ `` \dotplus " $, defined by $ n_1 \dotplus n_2 = n_1+n_2 $ if $ n_1+ n_2 < 4n, $ and $ n_1 \dotplus n_2 = n_1 +n_2-4n, $ if $ n_1 + n_2 \geq 4n, $ for all $ n_1, n_2\in \mathcal{N} $. For any $ j\in \mathcal{N} $, we will denote $ j \dotplus n $ by $ p $, $ j \dotplus 2n $ by $ q $, $ j \dotplus 3n $ by $ \mathbf{p} $ and $ j \dotplus 4n $ by $ \mathbf{q} $. It is trivial to see that $ \mathbf{q} = j $ but we use this notation only to make our theorems look more convenient. Let us recall that a polyhedron $ Q $ is said to be a face of the polyhedron $ P $ if either $ Q = P $ or if  we  can  write $ Q = P \cap \delta M $, where $ M $ is  a  closed  half-space in $ \mathbb{X} $ containing $ P ,$ and $ \delta M $ denotes the boundary of $ M $. If $ dim~Q = i, $ then $ Q $ is called an $ i $-face of $ P $. $ (n-1) $-faces of $ P $ are called facets of $ P $ and $ 1 $-faces of $ P $ are called edges of $ P. $ Let us now characterize the left-symmetric points and the right symmetric points of a two-dimensional polyhedral Banach space whose unit sphere is a regular polygon with $ 4n $ sides.

\begin{theorem}\label{4k gon left-symmetric}
Let $ \mathbb{X} $ be a two-dimensional real polyhedral Banach space such that $ S_{\mathbb{X}} $ is a regular polygon with $ 4n $ sides, where $ n \in \mathbb{N}. $ Then $ x \in S_{\mathbb{X}} $ is a left-symmetric point in $ \mathbb{X} $ if and only if $ x $ is the midpoint of some edge of $ S_{\mathbb{X}}. $
	
\end{theorem}

\begin{proof}

Up to an isomorphism, we may assume that the vertices of $ B_\mathbb{X} $ are $ v_j=(\cos \frac{2\pi j}{4n},\sin \frac{2\pi j}{4n}) $, $ j\in \mathcal{N} $. Let $ F_j $ be the facet of $ S_{\mathbb{X}} $ containing the vertices $ v_j $ and $ v_{j\dotplus 1} $. Let us first prove the necessary part of the theorem. Suppose, $ u\in S_{\mathbb{X}} $ is a left-symmetric point in $ \mathbb{X} $. Without loss of generality, we may assume that $ u\in F_j $, for some $ j\in \mathcal{N} $. Let $ f_j $ be the supporting functional corresponding to the facet $ F_j $. Then\\

$ f_j(x,y)= \{x\cos \frac{\pi (2j+1)}{4n}+y\sin \frac{\pi (2j+1)}{4n}\}\sec\frac{\pi}{4n} ~~$ for all $~~ (x,y)\in \mathbb{X} $.\\

It can be easily checked that $ kerf_j\cap S_\mathbb{X} = \{\frac{v_p+v_{p\dotplus 1}}{2}, \frac{v_{\mathbf{p}}+v_{\mathbf{p}\dotplus 1}}{2} \} $. Therefore, $ \{\frac{v_p+v_{p\dotplus 1}}{2}, \frac{v_{\mathbf{p}}+v_{\mathbf{p}\dotplus 1}}{2} \}\subseteq u^{\perp}\cap S_{\mathbb{X}} $. Since $ \frac{v_p+v_{p\dotplus 1}}{ 2 } $ is the middle point of the edge $ F_p $, $ \frac{v_p+v_{p\dotplus 1}}{ 2 } $ is smooth. Now, using analogous techniques as above, we obtain $  (\frac{v_p+v_{p\dotplus 1}}{ 2 })^{\perp}\cap S_{\mathbb{X}} = \{\frac{v_{q}+v_{q\dotplus 1}}{2}, \frac{v_{\mathbf{q}}+v_{\mathbf{q}\dotplus 1}}{2} \} = \{\frac{v_{q}+v_{q\dotplus 1}}{2}, \frac{v_j+v_{j\dotplus 1}}{2} \} $. Since we have assumed that $ u\in F_j $ is a left-symmetric point, we get that $ u $ must be of the form $ \frac{v_j+v_{j\dotplus 1}}{2} $. In other words, $ u $ must be the midpoint of the edge $ F_j $. We now prove the sufficient part of the theorem. Let $ u $ be the midpoint of some edge $ F_j $, $ j\in \mathcal{N} $. Then $ u^{\perp}\cap S_{\mathbb{X}} = \{\frac{v_p+v_{p\dotplus 1}}{2}, \frac{v_{\mathbf{p}}+v_{\mathbf{p}\dotplus 1}}{2} \} $. Clearly, $ \frac{v_p+v_{p\dotplus 1}}{2}, \frac{v_{\mathbf{p}}+v_{\mathbf{p}\dotplus 1}}{2} $ are the middle points of the edges $ F_p $ and $ F_{\mathbf{p}} $ respectively. Once again, using above arguments, it follows that $ u \in (\frac{v_p+v_{p\dotplus 1}}{2})^{\perp}\cap (\frac{v_{\mathbf{p}}+v_{\mathbf{p}\dotplus 1}}{2})^{\perp} $. This completes the proof of the theorem. 
\end{proof}

\begin{theorem}\label{4k gon right-symmetric}
Let $ \mathbb{X} $ be a two-dimensional real polyhedral Banach space such that $ S_{\mathbb{X}} $ is a regular polygon with $ 4n $ sides, where $ n \in \mathbb{N}. $ Then $ x \in S_{\mathbb{X}} $ is a right-symmetric point in $ \mathbb{X} $ if and only if $ x $ is an extreme point of $ B_{\mathbb{X}}. $
	
\end{theorem}

\begin{proof}

Up to an isomorphism, we may assume that the vertices of $ B_\mathbb{X} $ are $ v_j=(\cos \frac{2\pi j}{4n},\sin \frac{2\pi j}{4n}) $, $ j\in \mathcal{N} $. Let $ F_j $ be the facet of $ S_{\mathbb{X}} $ containing the vertices $ v_j $ and $ v_{j\dotplus 1} $. Let us first prove the necessary part of the theorem. In order to prove so, we will show that any smooth point of $ S_\mathbb{X} $ is not a right-symmetric point in $ \mathbb{X} $. Let $ u\in S_\mathbb{X} $ be a smooth point. Without loss of generality, suppose $ u\in F_j $ for some $ j\in \mathcal{N} $. Then by the previous theorem, $ u^{\perp}\cap S_{\mathbb{X}}=\{\frac{v_p+v_{p\dotplus 1}}{2}, \frac{v_{\mathbf{p}}+v_{\mathbf{p}\dotplus 1}}{2} \} $, namely, the middle points of the edges $ F_p $ and $ F_{\mathbf p} $ respectively. Clearly, $ F_p\cap (u^{+}~\setminus~ u^{\perp}) $ and $ F_{\mathbf p}\cap (u^{+}~\setminus~u^{\perp}) $ are nonempty and contain smooth points of $ B_\mathbb{X} $. Let $ w $ and $ w' $ be two smooth points of $ B_\mathbb{X} $ lying in $ F_p\cap (u^{+}~\setminus~u^{\perp}) $ and $ F_{\mathbf p}\cap (u^+~\setminus~u^{\perp}) $ respectively. Then it is easy to see that either, $ u\in w^{-} $, or $ u\in w'^{-} $.  In either occasion, $ w $ and $ w' $ are contained in $ u^{+}~\setminus~u^{\perp} $. Hence, by Theorem \ref{right-symmetric characterization}, $ u $ is not right-symmetric. Therefore, if $ u\in S_\mathbb{X} $ is right-symmetric point in $ \mathbb{X} $ then $ u $ must be an extreme point of $ B_\mathbb{X} $. We now prove the sufficient part of the theorem. Suppose $ u\in S_\mathbb{X} $ is an extreme point of $ B_\mathbb{X} $. Without loss of generality, we may assume that $ u=v_0 $. Clearly, $ v_0 \in F_0 \cap F_{4n-1} $. Let $ f_0 $ and $ f_{4n-1} $ be the supporting functionals corresponding to the facets $ F_0 $ and $ F_{4n-1} $ respectively. It is straightforward to check that $ ker~f_0 \cap S_\mathbb{X} = \{\frac{v_n+v_{n \dotplus 1}}{2}, \frac{v_{3n}+v_{3n\dotplus 1}}{2}\} $ and $ ker~f_{4n-1}\cap S_\mathbb{X} = \{\frac{v_n+v_{n-1}}{2}, \frac{v_{3n}+v_{3n-1}}{2}\} $. Let $ x_1=\frac{v_n+v_{n-1}}{2} $ and let $ x_2=\frac{v_{3n}+v_{3n\dotplus 1}}{2} $. Consider the sets\\

$ \mathcal{S}_0=\{ \alpha x_1 +\beta x_2 : \| \alpha x_1 + \beta x_2 \|=1, \alpha > 0,\beta > 0 \} $,\\

$ \mathcal{D}_0= S_\mathbb{X}\setminus \mathcal{S}_0 $.\\

\noindent Now, it is easy to see that\\

$ \mathcal{S}_0 = (v_0^{+}~\setminus v_0^{-})\cap S_\mathbb{X}  = \{ z \in S_\mathbb{X}: f_0(z) > 0 \} \cap \{ z \in S_\mathbb{X}: f_{4n-1}(z)> 0 \} , $\\

$  \mathcal{D}_0 = v_0^-\cap S_{\mathbb{X}} = \{ z \in S_{\mathbb{X}}: f_0(z)\leq 0 \}\cup \{ z\in S_{\mathbb{X}}: f_{4n-1}(z)\leq 0 \} $.\\

\noindent In other words, $ (v_0^{+}~\setminus v_0^{-})\cap S_\mathbb{X} $ is the shorter arc between $ x_1 $ and $ x_2 $ excluding $ x_1 $ and $ x_2 $ on $ S_\mathbb{X} $, and $ v_0^-\cap S_{\mathbb{X}} $ is the closed longer arc between $ x_1 $ and $ x_2 $ on $ S_\mathbb{X} $. In similar spirit, for any extreme point $ v_r $ of $ B_\mathbb{X} $, we define $ \mathcal{S}_r = (v_r^{+}~\setminus v_r^{-})\cap S_\mathbb{X} $ and $ \mathcal{D}_r = v_r^{-}\cap S_\mathbb{X} $ . It is easy to see that $ \mathcal{S}_0\cap E_\mathbb{X} = \{ v_{3n \dotplus 1}, v_{3n \dotplus 2}, \dots , v_0, v_1, \dots ,  v_{n-1}\} $ $ ( \{ v_0\} $ for $ n= 1 ) $ and $ \mathcal{D}_0\cap E_\mathbb{X}=\{ v_{n}, v_{n \dotplus 1}, \dots ,  v_{3n}\} $. We shall prove that $ v_0\notin w^{-} $ for all $ w\in \mathcal{S}_0 $. We observe that for any smooth point $ \kappa \in B_\mathbb{X} $ lying in $ F_j $, $ \kappa^{-} $ is properly contained in both $ v_j^{-} $ and $ v_{j\dotplus 1}^{-} $. Therefore, to prove $ v_0\notin w^{-} $ for every $ w\in \mathcal{S}_0 $, it is sufficient to prove $ v_0\notin v_j^{-}  $ for every $ v_j\in \mathcal{S}_0 $. It is evident that $ v_0\in v_r^{-} $ if and only if $ \mathcal{D}_r\cap E_\mathbb{X} $ contains $ v_0 $. Now, it is easily revealed that $ \mathcal{D}_r\cap E_\mathbb{X}=\{v_{r\dotplus s}\}_{s=n}^{3n} $. We observe that $ (r\dotplus s)=0 $ if and only if $ r\in \{n , n+1 , \dots , 3n \} $ i.e., $ v_r\in \mathcal{D}_0 $. Since $ (\mathcal{S}_0\cap E_{\mathbb{X}})\cap(\mathcal{D}_0\cap E_{\mathbb{X}}) = \emptyset $, we get that $ v_0\notin w^{-} $ for all $ w\in \mathcal{S}_0 $. In other words, $ v_0\in w^{-} $ implies $ w\in v_0^{-} $. In similar fashion, it can be shown that $ v_0\in w^{+} $ implies $ w\in v_0^{+} $. Therefore, by Theorem \ref{right-symmetric characterization}, $ v_0 $ is right-symmetric and this completes the proof of the theorem.

\end{proof}

Let $ \mathcal{M} = \{ 0, 1, \dots , 2n-1 \} $.  We equip $ \mathcal{M} $ with a binary operation $ `` \dotplus " $, defined by $ n_1 \dotplus n_2 = n_1+n_2, ~ \textit{if} ~ n_1 + n_2 < 2n, $ and $ n_1 \dotplus n_2 = n_1 +n_2 -2n, ~\textit{if} ~ n_1 + n_2 \geq 2n,$ for all $ n_1, n_2 \in \mathcal{M} $. Let $ \mathbb{Z} $ be a three-dimensional real polyhedral Banach space such that its unit sphere $ S_\mathbb{Z} $ is a right prism having a regular $ 2n $-gon as its base.Up to an isomorphism, we may assume that the vertices of $ B_\mathbb{X} $ are $ v_{\pm j}=(\cos\frac{2j\pi}{2n}, \sin\frac{2j\pi}{2n},\pm 1) $, $ j\in \mathcal{M} $. Clearly, the points $ \frac{(v_j+v_{-j})}{2} $, $ j\in \mathcal{M} $, are co-planer. Let us consider the plane containing $ \frac{(v_j + v_{-j})}{2} $,  $ j\in \mathcal{M} $. We observe that this plane is uniquely determined and we name it as the $ \mathbb{X} $-plane. Let us denote any rectangular facet of $ S_\mathbb{Z} $ containing the verices $ v_{\pm j} $ and $ v_{\pm (j\dotplus 1)} $ by $ V_j $. The plane $ |z|=1 $ intersects $ S_\mathbb{Z} $ in two $ 2n $-gonal facets. Let us denote these facets by $ \pm U $. In particular, $ S_\mathbb{Z} $ contains two types of facets,\\

\noindent (a) Facets $ (V_j)_{j=0}^{2n-1} $ which intersect the $ \mathbb{X} $-plane.\\
\noindent (b) Facets $ \pm U $ which are parallel to the $ \mathbb{X} $-plane.\\

\noindent In our context the facets of type-(a) will be called vertical facets and the facets of type-(b) will be called horizontal facets. The edges of rectangular facets that intersect the $ \mathbb{X} $-plane will be called vertical edges and the edges  of rectangular facets that are parallel to the $ \mathbb{X} $-plane will be called horizontal edges. Therefore, each vertical edge is contained in the supporting planes of vertical facets $ V_j $ and $ V_{j\dotplus 1} $ for some $ j\in \mathcal{M} $ and each horizontal edge is contained in the supporting planes of horizontal facet $ U $ (or $ -U $) and vertical facet $ V_j $ for some $ j\in \mathcal{M}. $ Now, $ \mathbb{X} = \{ (x, y, 0) \in \mathbb{Z} \} $. Let $ \mathbb{Y} = \{ (0, 0, z) \in \mathbb{Z} \} $. Let us equip $ \mathbb{X} $ and $ \mathbb{Y} $ with such norms that $ S_\mathbb{X} $ is a regular $ 2n $-gon and extreme points of $ B_\mathbb{X} $ are $ (\cos\frac{2j\pi}{2n}, \sin\frac{2j\pi}{2n}),~ j\in \mathcal{M} $ and $ B_\mathbb{Y} = [-1, 1] $. It is easy to see that $ \mathbb{Z} = \mathbb{X} \bigoplus_\infty \mathbb{Y} $. The points $ \pm 1 $ in $ S_{\mathbb{Y}} $ are both left-symmetric and right-symmetric in $ \mathbb{Y}. $ When $ n $ is even, applying Theorem \ref{left-symmetric} and Theorem \ref{4k gon left-symmetric}, we observe that a point in $ S_\mathbb{Z} $ is left-symmetric in $ \mathbb{Z} $ if and only if it is the centroid of some vertical facet or it is the centroid of some horizontal facet in $ S_{\mathbb{Z}} $. Similarly, applying Theorem \ref{right-symmetric} and Theorem \ref{4k gon right-symmetric}, we observe that a point in $ S_{\mathbb{Z}} $ is a right-symmetric point in $ \mathbb{Z} $ if and only if it is an extreme point of $ B_\mathbb{Z} $. Whenever $ n $ is odd, $ \mathbb{X} $ becomes a Radon plane. Now, applying Theorem \ref{left-symmetric}, we observe that a point in $ S_{\mathbb{Z}} $ is a left-symmetric point in $ \mathbb{Z} $ if and only if it is the centroid of some horizontal facet in $ S_{\mathbb{Z}} $ or it lies in $ S_\mathbb{X}.$ Similarly, applying Theorem \ref{right-symmetric}, we observe that a point in $ S_{\mathbb{Z}} $ is a right-symmetric point in $ \mathbb{Z} $ if and only if it lies in the horizontal edges of $ S_\mathbb{Z} $. We rewrite these simple facts in form of two separate theorems below to end the present article.

\begin{theorem}\label{left right symmetric 4n-prism}
Suppose $ \mathbb{Z} $ is a three-dimensional real polyhedral Banach space such that its unit sphere is a right prism having a regular $ 4n $-gon as its base. Then $ z\in S_\mathbb{Z} $ is a left-symmetric point in $ \mathbb{Z} $ if and only if it is either a centroid of a vertical facet in $S_\mathbb{Z} $ or a centroid of a horizontal facet in $ S_\mathbb{Z} $. Also, $ z\in S_\mathbb{Z} $ is a right-symmetric point in $ \mathbb{Z} $ if and only if it is an extreme point of $ B_\mathbb{Z} $.
\end{theorem}

\begin{theorem}\label{left right symmetric 4n+2-prism }
Suppose $ \mathbb{Z} $ is a three-dimensional real polyhedral Banach space such that its unit sphere is a right prism having a regular $ (4n+2) $-gon as its base. Then $ z\in S_\mathbb{Z} $ is a left-symmetric point in $ \mathbb{Z} $ if and only if it is either a centroid of a horizontal facet in $ S_\mathbb{Z} $ or of the form $ (x, y , 0) $. Also, $ z\in S_\mathbb{Z} $ is a right-symmetric point in $ \mathbb{Z} $ if and only if it lies on a horizontal edge of $ S_\mathbb{Z} $.
\end{theorem}

\end{document}